\documentclass[a4paper, 11pt]{amsart}
\usepackage{amsmath,amssymb,amscd}
\usepackage[colorlinks=true, allcolors=blue]{hyperref}
\usepackage{url}
\usepackage[latin1]{inputenc}
\usepackage{graphicx}
\usepackage{tikz-cd}
\usepackage{enumerate}
\newtheorem{theorem}{Theorem}[section]

\newtheorem{proposition}[theorem]{Proposition}

\newtheorem{lemma}[theorem]{Lemma}
\theoremstyle{definition}
\newtheorem{definition}[theorem]{Definition}
\newtheorem{example}[theorem]{Example}
\newtheorem{remark}[theorem]{Remark}

\usepackage{color}

\title{Property \(P_{\text{naive}}\) for big mapping class groups}
\author{Tianyi Lou}
\date{\today}
\address{Universit\'e C\^ote d'Azur,
CNRS, LJAD (UMR CNRS 7351), Parc Valrose, 06108
Nice Cedex 2, France} 
\address{Institut Fourier, UMR 5582, Laboratoire de Math\'ematiques, Universit\'e Grenoble Alpes, CS 40700, 38058 Grenoble cedex 9, France} 
\email{tianyi.lou@univ-grenoble-alpes.fr/tianyi.lou@unice.fr}


\begin{document}

\begin{abstract}
We prove that the mapping class group of any infinite type surface containing a nondisplaceable subsurface of finite type satisfies the property \(P_{\text {naive }}\): for any finite collection of non-trivial elements $h_{1},h_{2},$ $\cdots, h_{n}$, there exists another element $g\neq 1$ of infinite order such that $\langle g, h_{i}\rangle \cong \langle g \rangle * \langle h_{i} \rangle$, for all $i$.  
\end{abstract}

\maketitle

\section{Introduction}
A group $G$ is said to have the \emph{property \(P_{\text {naive}}\)} if, for any finite subset \(F\) of \(G \backslash\{1\}\), there exists an element \(g \in G\) of infinite order such that, for every \(h \in F\), the subgroup \(\left\langle h, g\right\rangle\) of \(G\) is isomorphic to the free product \(\langle h\rangle *\left\langle g\right\rangle\). This witnesses a high degree of dynamical flexibility for the possible subgroups. This property was introduced by Bekka, Cowling and de la Harpe \cite{BCD95}, who observed that if a discrete group $G$ has the property \(P_{\text {naive }}\), then the reduced $C^{*}$-algebra $C^{*}_{r}(G)$ of $G$ is simple.
For a free group, this is a consequence of the Nielsen-Schreier theorem, which states that every subgroup of a free group is free.

Bekka, Cowling and de la Harpe proved that Zariski-dense discrete subgroups of connected simple groups of rank one with trivial center have property $P_{\text{naive}}$ \cite[Theorem 3]{BCD95}. 
A large class of examples is provided by acylindrically hyperbolic groups: as shown by Abbott and Dahmani \cite{AD18}, any such group with no nontrivial finite normal subgroups has property $P_{\text{naive}}$. 
This subsumes earlier results for relatively hyperbolic groups \cite{AM07} and for some groups acting on $\operatorname{CAT}(0)$ cube complexes \cite{KS16}. Chatterji and Martin gave criteria making such cubical actions acylindrically hyperbolic \cite{CM16}. 
See also \cite{MO15} for further acylindrically hyperbolic examples.

Surfaces are of \emph{finite type} if their fundamental group is finitely generated and of \emph{infinite type} otherwise. The mapping class group $\operatorname{Map}(S)$ is acylindrically hyperbolic if $S$ is of finite type (provided $\pi_{1}(S)$ is non-elementary) \cite{B08}. In the case $S$ is of infinite type, \(\operatorname{Map}(S)\) is never acylindrically hyperbolic \cite{BG18}. It nonetheless exhibits certain acylindrical-like features through mapping classes that preserve finite type subsurfaces. Mann and Rafi introduced the concept of the \emph{nondisplaceable subsurface} (see Definition \ref{nondis}) to study the geometric properties of \(\operatorname{Map}(S)\). This is a topological restriction because it is not automatic for infinite type surfaces. In fact, Horbez, Qing and Rafi \cite{HQR20} proved that if $S$ contains a nondisplaceable subsurface of finite type, then there exists a continuous, nonelementary, isometric action of \(\operatorname{Map}(S)\) on a hyperbolic space, constructed using the projection complex machinery of \cite{BBF15}. 

Recall that the \emph{complexity} of a connected orientable surface $S$ is defined as \(\xi(S)=3 g(S)+|E(S)|-3\). Our main result is the following.

\begin{theorem} \label{main}
Let $S$ be a connected orientable surface with positive complexity, and assume that $S$ contains a nondisplaceable connected subsurface of finite type. Then $\operatorname{Map}(S)$ has the property $P_{\text {naive}}$.   
\end{theorem}

We first explain the strategy of the proof of Theorem~\ref{main}. To prove the main result, we consider $n$ arbitrary mapping classes $h_{1}, h_{2}, \cdots, h_{n}$ of $\operatorname{Map}(S)$ and find an infinite order element $g$ such that, for each $ 1\leq i \leq n$, the group generated by $h_{i}$ and $g$ is a free product $\langle h_{i} \rangle * \langle g \rangle$. 

We will find the element $g$ by taking a nondisplaceable subsurface $K$ which intersects the support of the $h_{i}$ and selecting a \(K\)-pseudo-Anosov mapping class. An element \(f \in \operatorname{Map}(S)\) is \emph{\(K\)-pseudo-Anosov} if \(f\) preserves the isotopy class of the subsurface \(K \) of $S$ and, denoting by \(\widehat{K}\) a surface obtained from \(K\) by gluing a once-punctured disk on every boundary component of \(K\), the mapping class \(f\) induces a pseudo-Anosov mapping class of \(\widehat{K}\).

Horbez, Qing and Rafi in \cite[Theorem 2.9]{HQR20} already noticed that such an element is WWPD loxodromic (see Definition \ref{DefWWPD}). To prove $g$ and $h$ span a free product \(\langle h\rangle *\left\langle g\right\rangle\), we employ different approaches guided by the classification of elements \(h \in F \subset \operatorname{Map}(S)\) acting on a hyperbolic space.

\subsection*{Acknowledgments:}

I am deeply grateful to my advisors, Indira Chatterji and Fran\c{c}ois Dahmani, for their guidance, encouragement, and many illuminating discussions. I would also like to extend my appreciation to Juliette Bavard for her insightful conversations during a visit. Appreciation also goes to Yusen Long and Yibo Zhang for their helpful comments. I also thank the anonymous referees for their careful reading and helpful suggestions. Finally, thanks to the China Scholarship Council for their fellowship support.

\section{Preliminary}

\subsection{Surfaces and mapping class groups}\label{2.1}

In this context, a $\emph{surface}$ $S$ refers to an oriented, connected, second-countable, Hausdorff manifold of dimension two. Unless otherwise stated, $S$ has no boundary component.

By a theorem of Richards \cite{Ric63}, connected, orientable surfaces $S$ are classified up to homeomorphism by the triple $\left(g(S), E(S), E^{g}(S)\right)$, where $g(S)$ is the genus of $S$, $E(S)$ is the space of ends of $S$ and \(E^{g}(S)\subseteq E(S)\) is the subset of ends accumulated by genus, i.e. every neighborhood \(U\) of \(e\in E^{g}(S)\) has infinite genus.

A \emph{subsurface} $K\subset S$ is a closed subset of $S$ such that $K$ is a surface and $\partial K$ is a finite union of pairwise disjoint simple closed curves contained in $\mathrm{int}(S)$. A simple closed curve $c\subset S$ is \emph{essential} if it does not bound a disk or a once-punctured disk.
We call $K$ \emph{essential} if every component of $\partial K$ is an essential
simple closed curve in $S$.  

Let \(\operatorname{Homeo}^{+}(S)\) be the group of all orientation-preserving homeomorphisms \(S \rightarrow S\). We denote by \(\operatorname{Homeo}_{0}(S)\) the connected component of the identity in \(\operatorname{Homeo}^{+}(S)\). The $\emph{mapping~class~group}$ of \(S\) is
$$
\operatorname{Map}(S):=\operatorname{Homeo}^{+}(S) / \operatorname{Homeo}_{0}(S).
$$
The group $\operatorname{Map}(S)$ is equipped with the quotient topology of the compact-open topology on the group \(\operatorname{Homeo}^{+}(S)\).

\subsection{Nondisplaceable subsurfaces}
The notion of nondisplaceable subsurface of $S$ was introduced in the work of Mann and Rafi \cite{MR23}.

\begin{definition}\label{nondis}
A connected, finite type subsurface \(K\) of a surface $S$ is \emph{nondisplaceable} if $f(K) \cap K \neq \emptyset$
for all \(f \in \operatorname{Homeo}(S)\). A non-connected subsurface \(K\) of a surface $S$ is nondisplaceable if, for every $f \in \operatorname{Homeo}(S)$ there are connected components $K_{i}, K_{j}$ of $K$ such that \(f(K_{i}) \cap K_{j} \neq \emptyset\).
\end{definition}

\begin{example}
Generally, if $S$ has positive finite genus, any subsurface $K\subset S$ with $g(K)=g(S)$ is nondisplaceable (middle).
Using the space of ends gives many further examples: if $S$ is a surface with $|E(S)|\ge 3$, then any subsurface $K$ that separates the points of $E$ into distinct complementary components is nondisplaceable (left). The subsurface with all punctures in the punctured Loch Ness Monster surface is nondisplaceable (right).
\end{example}

\begin{figure}[htbp]
\centering
\includegraphics[height=3cm,width=12cm]{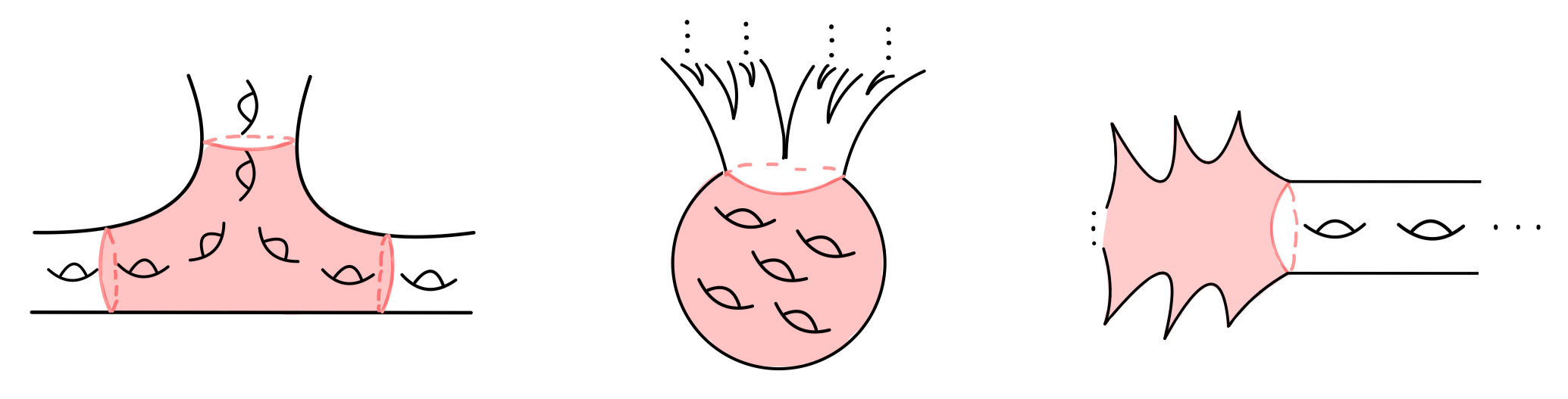}
\caption{Examples of nondisplaceable subsurfaces.}
\label{equator}
\end{figure}

Notice that if a connected subsurface $K$ contains a nondisplaceable subsurface $K^{\prime}$, then $K$ is also nondisplaceable. 

In the rest of this subsection, we prove auxiliary lemmas on nondisplaceable subsurfaces of an infinite type surface $S$ to facilitate the discussion of mapping class groups. We define the subgroup $\operatorname{Map}(K, \partial K)$ of $\operatorname{Map}(S)$ relative to $\partial K$ to be those mapping classes that fix the boundary $\partial K$ pointwise and, if $K$ has punctures, also fix them pointwise.

\begin{lemma}\label{inj}
Let $S$ be a connected, orientable surface containing a nondisplaceable subsurface $K^{\prime}$ of finite type. Then there exists a nondisplaceable connected, orientable, essential finite type subsurface $K$ containing $K^{\prime}$, such that the inclusion $K \subset S$ induces an inclusion of $\operatorname{Map}(K, \partial K)$ into $\operatorname{Map}(S)$. 

\end{lemma}

\begin{proof}

Every inclusion \(i:K^{\prime} \rightarrow S\) induces a natural mapping
$$
i_{*}: \operatorname{Map}(K^{\prime}, \partial K^{\prime}) \rightarrow \operatorname{Map}(S).
$$
If \([h]\) is the isotopy class of a homeomorphism of \(K^{\prime}\) fixing $\partial K^{\prime}$ pointwise, then \(i_{*}([h])\) is represented by extending
\(h\) to \(S\) using the identity on \(S \backslash K^{\prime}\). The kernel of \(i_{*}\) is generated by
$$
\left\{A_{1}, \ldots, A_{r}, B_{1}^{-1} B_{1}^{\prime}, \ldots, B_{s}^{-1} B_{s}^{\prime}\right\},
$$
where $A_{i}$ are the Dehn twists corresponding to the curves that bound a once-punctured disk for $i=1, \cdots, r$ and $B_{j}, B_{j}^{\prime}$ are the Dehn twists corresponding to the curves that bound an annulus for $j=1, \cdots, s$ \cite[Theorem 4.1]{PR00}. 

If $K^{\prime}$ is connected, let $K$ be a subsurface of $S$ consisting of $K^{\prime}$ and all the once-punctured disks and annuli of $S \backslash K^{\prime}$. Then, $K$ is also a nondisplaceable subsurface of $S$. The natural mapping $
i_{*}: \operatorname{Map}(K, \partial K) \rightarrow \operatorname{Map}(S)$ induced by inclusion \(i:K \rightarrow S\) is injective \cite[Corollary 4.2]{PR00}(or see \cite[Theorem 3.18]{FM11}).

If $K^{\prime}$ is not connected, then there exists a connected subsurface $K^{\prime\prime}$ of finite type of $S$ that contains $K^{\prime}$, since every connected surface can be exhausted by connected subsurfaces of finite type. Let $K$ be a subsurface of $S$ formed by $K^{\prime\prime}$ together with all once-punctured disks and annuli in $S \backslash K^{\prime\prime}$. This construction completes the proof.
\end{proof}

 The following lemma allows us to construct a larger nondisplaceable subsurface of finite type in any infinite type surface containing a nondisplaceable subsurface of finite type. 

\begin{lemma}\label{newlemma}
    Let $S$ be a surface and $K^{\prime} \subset S$ be a nondisplaceable subsurface of finite type. Let $c_{1}, c_{2}, \cdots, c_{r}$ be a family of essential simple closed curves on $S$. Then there exists a connected nondisplaceable subsurface $K$ of $S$ that contains $c_{1}, c_{2}, \cdots, c_{r}$. Moreover, we can choose $K$ to be essential.
\end{lemma}

\begin{proof}
    If $c_{1}, c_{2}, \cdots, c_{r}$ are all contained in $K^{\prime}$, take $K=K^{\prime}$. 
    Without loss of generality, suppose that $c_{1}$ is not contained in $K^{\prime}$. Fixing a pants decomposition of $S$, we can always find $K_{1}$ consisting of finitely many pants such that it contains $K^{\prime}$ and $c_{1}$. By induction, take $K= \bigcup_{i=1}^{r}K_{i}$ to complete the proof. To make $K$ essential, add the finitely many missing disks, once-punctured disks, and annuli of $S$ bounded by components of $\partial K$. Therefore, we can choose $K$ to be essential.
\end{proof}

\begin{lemma}\label{h}
Let $S$ be a connected orientable surface with positive complexity, and assume that $S$ contains a nondisplaceable connected subsurface of finite type. Let $\{h_{1}, h_{2}, \cdots, h_{n}\}$  be a collection of non-trivial elements in $\operatorname{Map}(S)$. Then there exists a nondisplaceable surface $K$ of finite type such that the mapping class $h_i$ has no representative restricting to the identity on $K$. 
\end{lemma}

\begin{proof}
As the Alexander method ensures that any non-trivial mapping class must move at least one simple closed curve, there exists a simple closed curve $c_i$ such that $h_i(c_i)$ is not homotopic to $c_i$.  Recall then that, by Lemma \ref{newlemma} we may replace $K$ once and for all so that it contains all the curves $c_i$. 

Fix $i$. If $h_i$ preserves $K$ and admits a representative $f$ with $f|_K=\mathrm{Id}$, then in particular
$f(c_i)=c_i$ since $c_i\subset K$, hence $h_i(c_i)$ is homotopic to $c_i$, a contradiction. Therefore, whenever $h_i$ preserves $K$, no representative of $h_i$ can restrict to the identity on $K$. On the other hand, if $h_i$ does not preserve $K$, then it cannot have a representative restricting to the identity on $K$.
   
\end{proof}

\textbf{Convention}: unless stated otherwise, we always assume that $S$ contains a nondisplaceable subsurface $K$ of finite type such that no component of $S \backslash K$ is an annulus whose boundary components are both boundary components of $K$, a one-punctured disk, or a disk. When we use Lemma \ref{newlemma} and Lemma \ref{h}, the convention will still be satisfied.

\subsection{Quasi-trees of metric spaces and WWPD}

We follow the projection data framework of \cite{BBF15}. Let $\mathbf{Y}$ be a set. For each $Y \in \mathbf{Y}$, let $\mathcal{C}(Y)$ denote an associated geodesic metric space. For \(X, Y \in \mathbf{Y}\), with $X \neq Y$, we are given a non-empty set $\pi_{Y}(X)\subseteq \mathcal{C}(Y)$, called the \emph{projection} of $X$ on $Y$. For any $X, Y, Z \in \mathbf{Y}$ with $X\neq Y$ and $Z\neq Y$, the following \emph{projection distance} is defined:
\[
d_{Y}(X, Z)= \operatorname{diam}_{\mathcal{C}(Y)}\bigl(\pi_{Y}(X) \cup \pi_{Y}(Z)\bigr).
\]

\begin{definition}
    The collection $\{(\mathcal{C}(Y), \pi_{Y})\}_{Y \in \mathbf{Y}}$ satisfies the \emph{projection axioms} for a real projection constant $\theta \geq 0$ if
    \begin{itemize}
    \item [(P1)] For any two distinct $X, Y \in \mathbf{Y}$, $\operatorname{diam}\bigl(\pi_{Y}(X)\bigr) \leq \theta;$
    \item [(P2)] For any $X, Y, Z \in \mathbf{Y}$, if $d_{Y}(X, Z) > \theta$, then $d_{X}(Y, Z) \leq \theta;$
    \item [(P3)] For any two distinct $X, Y \in \mathbf{Y}$, the set $\{Z \in \mathbf{Y} \mid d_{Z}(X, Y) > \theta\}$ is finite.
\end{itemize}
\end{definition}

Bestvina, Bromberg, and Fujiwara \cite{BBF15} define a connected graph $\mathcal{P}_{\alpha}(\mathbf{Y})$, called \emph{projection complex}, for $\alpha \geq 2 \theta$ which has the set of vertices as $\mathbf{Y}$, and there is an edge between $X, Z \in \mathbf{Y}$ if and only if $d_{Y}(X, Z) \leq \alpha$ for every $Y \in \mathbf{Y}$. Then they build the space $\mathcal{C}_{\alpha}(\mathbf{Y})$ by replacing each vertex $Y \in \mathbf{Y}$ with $\mathcal{C}(Y)$ and joining points in $\pi_{Y}(X)$ with points $\pi_{X}(Y)$ by an edge of length $L(\alpha) >0$ whenever $X$ and $Y$ has an edge between them in $\mathcal{P}_{\alpha}(\mathbf{Y})$, i.e. $d_{\mathcal{P}_{\alpha}(\mathbf{Y})}(X, Y)=1$. This quasi-tree depends on the choice of $\alpha$ and $L$. We fix $L$ as a function of $\alpha$ and assume that $\alpha$ is sufficiently large, then each $\mathcal{C}(Y)$ will be isometrically embedded in $\mathcal{C}_{\alpha}(\mathbf{Y})$ (\cite[Lemma 4.2]{BBF15}).

We say that a group \(G\) acting on \(\mathbf{Y}\) preserves the projection structure if for each $Y \in \mathbf{Y}$ and $g \in G$ there are isometries \(F_{g}^{Y}: \mathcal{C}(Y) \rightarrow \mathcal{C}(g(Y))\) so that
\begin{itemize}
    \item [(1)] \(F_{g^{\prime}}^{g(Y)} F_{g}^{Y}=F_{g^{\prime} g}^{Y}\) for all \(g, g^{\prime} \in G, Y \in \mathbf{Y}\) and
    \item [(2)] \(g(\pi_{Y}(X))=\pi_{g(Y)}(g(X))\) for all \(g \in G\) and \(X, Y \in \mathbf{Y}\).
\end{itemize}
If a group $G$ acts on \(\mathbf{Y}\) in this way, then the projection distances are preserved, i.e. \(d_{g(X)}(g(Y), g(Z))=d_{X}(Y, Z)\)
for all \(X, Y, Z \in \mathbf{Y}\) and \(g \in G\). This action of $G$ on \(\mathbf{Y}\) is said to be \emph{metric-preserving}. Therefore, \(G\)
acts isometrically on \(\mathcal{C}_{\alpha}(\mathbf{Y})\). 

Recall that a geodesic metric space $X$ is said to be \emph{$\delta$-hyperbolic} if every geodesic triangle in $X$ is \emph{$\delta$-slim}, i.e. each of its sides is contained in the $\delta$-neighbourhood of the union of the other two sides. 
We say that $X$ is \emph{Gromov hyperbolic} if it is $\delta$-hyperbolic for some $\delta \geq 0$. Pick a basepoint $z$ in a $\delta$-hyperbolic space $X$. The \emph{Gromov product} of $x,y \in X$ is defined to be $(x, y)_{z} = \frac{1}{2}\bigl(d(x,z) + d(y,z) - d(x,y)\bigr)$. Let \(G\) be a group acting by isometries on a \(\delta\)-hyperbolic metric space \(X\). For an arbitrary point \(x \in X\), we denote by \(\Lambda (G) \subseteq \partial X\) the set of limit points of \(G\). That is, $\Lambda (G)$ is the set of accumulation points of any orbit of G on $\partial X$. Thus,
$$
\Lambda (G):=\left\{p \in \partial X \mid p=\lim _{n \rightarrow \infty} g_{n} x \text { for some sequence } g_{n} \in G\right\}
.$$
The limit set is independent of the choice of \(x\). Given an element $g \in G$, we denote $\Lambda(\langle g \rangle)$ simply by $\Lambda(g)$ and call it the limit set of $g$.

The standard classification of groups acting on hyperbolic spaces goes back to Gromov \cite[Section 8.2]{Gro87}. Let $G$ act on a $\delta$-hyperbolic space $X$ by isometries.
An element \(g \in G\) is called \emph{elliptic} if \(\Lambda(g)=\varnothing\) (equivalently, all orbits of
\(\langle g\rangle\) are bounded); \emph{parabolic} if \(|\Lambda(g)|=1\); and \emph{loxodromic} if \(|\Lambda(g)|=2\). 

The following statement follows directly from the Bestivina-Bromberg-Fujiwara construction.  

\begin{theorem}\cite{BBF15}\label{hyp}
 Assume that there exists a collection $\{(\mathcal{C}(Y), \pi_{Y})\}_{Y \in \mathbf{Y}}$ satisfying the projection axioms \textnormal{(P1), (P2)}, and \textnormal{(P3)} with projection constant $\theta \geq 0$ and $\alpha > 3 \theta$. Then 
  \begin{enumerate}
     \item $\mathcal{P}_{\alpha}(\mathbf{Y})$ is a quasi-tree.
     \item If for some $\delta \geq 0$, each $\mathcal{C}(Y)$ is $\delta$-hyperbolic, then $\mathcal{C}_{\alpha}(\mathbf{Y})$ is $\delta$-hyperbolic.
     \item If for some $\Delta \geq 0$, each $\mathcal{C}(Y)$ is a quasi-tree with bottleneck constant $\Delta$, then $\mathcal{C}_{\alpha}(\mathbf{Y})$ is a quasi-tree.
 \end{enumerate}  
\end{theorem}

For convenience, we will assume that $\alpha$ is sufficiently large and omit it from the notation.

We recall the notion of \emph{weakly properly discontinuous} (shortly, WPD) isometries of hyperbolic metric spaces introduced by Bestvina and Fujiwara in \cite{BF02}. A loxodromic element $g$ in a hyperbolic action is WPD roughly when the group action is acylindrical along the axis of $g$. The condition WWPD was introduced by Bestvina, Bromberg, and Fujiwara in \cite{BBF15}. It is a weakening of WPD which allows the loxodromic element $g$ to have a large centralizer.

\begin{definition}\label{DefWWPD}
Let $G$ be a group acting by isometries on a $\delta$-hyperbolic metric space $X$.
 We say that \(g \in G\) is WWPD if
 \begin{enumerate}
     \item \(g\) is loxodromic and $\Lambda(g)=\{g^{+}, g^{-}\} \subseteq \partial X$,
     \item if \( h_{n} \in G\) with \(h_{n}\left(g^{+}\right) \rightarrow g^{+}\)and \(h_{n}\left(g^{-}\right) \rightarrow g^{-}\), then there exists \(N>0\)
such that for all \(n \geq N\)
$$
h_{n}\left(g^{+}\right)=g^{+} \text {and } h_{n}\left(g^{-}\right)=g^{-}
$$    
 \end{enumerate}  
\end{definition}

Another equivalent definition is the following. An element $g\in G$ is WWPD
with respect to the $G$-action on a hyperbolic space $X$ if $g$ is loxodromic
and the $G$-orbit of its ordered pair of fixed points $g^{\pm}=(g^{-},g^{+})$
is a discrete subset of $(\partial X\times \partial X)\setminus\Delta$, where
$\Delta=\{(\xi,\xi)\mid \xi\in\partial X\}$ is the diagonal. Equivalently, the
$G$-orbit of the corresponding unordered pair $\{g^{-},g^{+}\}$ is discrete in
\[
\partial^{2}X := \bigl\{\{\xi,\eta\}\subset\partial X \mid \xi\neq \eta\bigr\},
\]
since $\partial^{2}X$ is the quotient of $(\partial X\times \partial X)\setminus\Delta$ by the involution $(\xi,\eta)\sim(\eta,\xi)$. See \cite{HM21}.

\subsection{Quasi-trees of curve graphs for big mapping class groups}
Let $S$ be a connected surface. The curve graph $\mathcal C(S)$ has vertices the isotopy classes of essential simple closed curves on $S$, and an edge joins two vertices if they admit disjoint representatives. Assume that S contains a connected nondisplaceable subsurface $K$ of finite type. Let $\mathcal{C}_{S}(K)$ be the graph whose vertices are the homotopy classes of simple closed curves on $S$ that have a representative contained in $K$ which is essential in $K$, where an edge joins two distinct isotopy classes if they can be realized by disjoint representatives in $S$. The curve graph $\mathcal{C}_{S}(K)$ is an induced subgraph of the curve graph $\mathcal{C}(S)$ of $S$. Also, $\mathcal{C}_{S}(K)$ only depends on the isotopy class of $K$: if $K$ and $K_{1}$ are isotopic, then there is a natural identification between $\mathcal{C}_{S}(K)$ and $\mathcal{C}_{S}(K_{1})$. We write $[K]$ for the isotopy class of $K$. Then we consider $\mathbf{Y}_{K}$ to be the orbit of $K$ under the metric-preserving action of $\operatorname{Map(S)}$. Each $K$ has an associated geodesic metric space $\mathcal{C}_{S}(K)$.

As \(K\) is nondisplaceable and of finite type, given any two non-isotopic subsurfaces \(K_{1}, K_{2} \in (\operatorname{Homeo}(S) \cdot K)\), at least one of the boundary components of \(K_{2}\) intersects the subsurface
\(K_{1}\) in an essential curve or arc (since $K$ is nondisplaceable we have $K_1\cap K_2\neq\varnothing$, and if
$\partial K_2\cap K_1=\varnothing$ then $K_1$ and $K_2$ are nested and isotopic, contradicting $K_1\neq K_2$). For each of these arcs, closing them up along $\partial{K_{1}}$ in potentially two different ways gives us a collection of essential simple closed curves in $K_{1}$. For \(K_{1}, K_{2} \in (\operatorname{Map}(S) \cdot K)\), we define a
projection
$$
\pi_{K_{1}}\left(K_{2}\right) \subseteq \mathcal{C}_{S}\left(K_{1}\right),
$$
to be the union of all essential simple closed curves and closed up arcs coming from $\partial K_2\cap K_1$. Then we define distances between subsurface projections. For \(K_{1}, K_{2}, K_{3} \in (\operatorname{Map}(S) \cdot K)\)
, define
$$
d_{K_{1}}(K_{2}, K_{3})=\operatorname{diam}_{\mathcal{C}_{S}(K_{1})}\left(\pi_{K_{1}}(K_{2}), \pi_{K_{1}}(K_{3})\right).
$$

\begin{proposition}\cite{BBF15, HQR20}
    Let $S$ be a connected orientable surface that contains a connected nondisplaceable subsurface $K$ of finite type. Then the family $\{(\mathcal{C}_{S}(K_{1}), \pi_{K_{1}})_{K_{1}\in \mathbf{Y}_{K}}\}$ satisfies the projection axioms \textnormal{(P1), (P2)}, and \textnormal{(P3)}.
\end{proposition}

\begin{theorem}\cite{HQR20}\label{WWPD}
Let \(S\) be a connected orientable surface with positive complexity, and assume that
\(S\) contains a nondisplaceable connected subsurface \(K\) of finite type.
Then there exists an unbounded $\delta$-hyperbolic space \(\mathbb{X}=\mathcal{C}(\mathbf{Y}_{K})\) equipped with a continuous nonelementary isometric action of \(\operatorname{Map}(S)\) such that every element of \(\operatorname{Map}(S)\) which is \(K\)-pseudo-Anosov is a WWPD loxodromic element for the \(\operatorname{Map}(S)\)-action on \(\mathbb{X}\).
\end{theorem}

\section{General results on the action of \texorpdfstring{\(\operatorname{Map}(S)\)}{Map(S)} on \texorpdfstring{\(\mathbb{X}\)}{X}} \label{3}

In this part, we will explore some useful facts regarding the continuous nonelementary isometric action of \(\operatorname{Map}(S)\) on \(\mathbb{X}\), where $\mathbb{X}=\mathcal{C}(\mathbf{Y}_{K})$ is an unbounded $\delta$-hyperbolic space
from Theorem \ref{WWPD} for a certain nondisplaceable surface.

 Let $S$ be a surface and $K \subset S$ be a subsurface of finite type obtained by removing finitely many pairwise disjoint separating simple closed curves. 
 The following lemma shows that the inclusion $K \subset S$ induces an inclusion of $\mathcal{C}(K)$ into $\mathcal{C}(S)$ and the image is precisely $\mathcal{C}_{S}(K)$.
\begin{lemma}\cite{HQR20}
    Let $S$ be a surface and $K \subset S$ be a nondisplaceable subsurface of finite type. Let $c$ and $c^{\prime}$ be two essential simple closed curves on $S$ that are contained and essential in $K$. If they are homotopic in $S$, then they are isotopic in $K$. 
    
    Moreover, if their isotopy classes have disjoint representatives in $S$, then they have disjoint representatives contained in $K$.
\end{lemma}

 Let $\widehat{K}$ denote the surface obtained from $K$ by gluing a once-punctured disk on every boundary component of $K$. The inclusion $K \hookrightarrow \widehat{K}$ induces a natural homomorphism $\operatorname{Map}(K) \rightarrow \operatorname{Map}(\widehat{K})$, whose kernel is free abelian, generated by Dehn twists about the boundary curves of $K$, by \cite[Proposition 3.9]{FM11}.  Let $\operatorname{Stab}_{\operatorname{Map}(S)}(K)$ be the subgroup of $\operatorname{Map}(S)$ made of all mapping classes that preserve the isotopy class of K and $\operatorname{Fix}_{\operatorname{Map}(S)}(K)$ be the subgroup of $\operatorname{Map}(S)$ made of all elements that have a representative $\psi \in \operatorname{Homeo}(S)$ such that $\psi(K)=K$ and $\psi|_{K}=\mathrm{Id}_{K}$. So there exists a homomorphism $\operatorname{Stab}_{\operatorname{Map}(S)}(K) \rightarrow \operatorname{Map}(\widehat{K})$, whose kernel is equal to $\operatorname{Fix}_{\operatorname{Map}(S)}(K)$ by \cite[Lemma 2.4]{HQR20}. Then the setwise stabilizer of $\mathcal{C}_{S}(K)$ is equal to $\operatorname{Stab}_{\operatorname{Map}(S)}(K)$ and the pointwise stabilizer of $\mathcal{C}_{S}(K)$ is equal to $\operatorname{Fix}_{\operatorname{Map}(S)}(K)$ for the action of $\operatorname{Map}(S)$ on $\mathcal{C}(S)$.

By Lemma \ref{inj}, the inclusion $K \subset S$ induces an inclusion of $\operatorname{Map}(K, \partial K)$ into $\operatorname{Map}(S)$. Let $h$ be an element of $\operatorname{Map}(S)$. Suppose that $h$ preserves the isotopy class of $K$ but the restriction on $K$ is not equal to the identity i.e. $h \in \operatorname{Stab}_{\operatorname{Map}(S)}(K) \backslash \operatorname{Fix}_{\operatorname{Map}(S)}(K)$. Therefore, we have always $h(\mathcal{C}_{S}(K))= \mathcal{C}_{S}(K)$ and $h|_{\mathcal{C}_{S}(K)} \neq  \mathrm{Id}_{\mathcal{C}_{S}(K)}$. The action of  $\operatorname{Map}(K)$ on $\mathcal{C}_{S}(K)$ is acylindrical proved by Bowditch in \cite{B08} and $\operatorname{Map}(K, \partial K)$ as a subgroup of $\operatorname{Map}(K)$ also acts acylindrically on $\mathcal{C}_{S}(K)$. Hence $\operatorname{Map}(K, \partial K)$ is an acylindrically hyperbolic group for sufficiently complex $K$ (the genus $g(K) \geq 2$ or the boundary components $b(K) > 3$). This condition can always be arranged, since the surface \(S\) is of infinite type. By \cite[Section 3.4]{FM11}, $\operatorname{Map}(K, \partial K)$ has no non-trivial finite normal subgroup. Therefore, $\operatorname{Map}(K, \partial K)$ has Property \(P_{\text {naive }}\) by the following theorem:

\begin{theorem}\cite{AD18}\label{ad}
Let \(G\) be an acylindrically hyperbolic group with no non-trivial finite normal subgroup. Then \(G\) has property \(P_{\text {naive. }}\)
\end{theorem}

Given an isometry \(g\) of a space \(X\), we say that a point $x$ is \emph{\(r\)-quasi-minimal}
for \(g\) if \(d(x, g(x)) \leq \inf _{y \in X} d(y, g(y))+r\). The set consisting of all \(r\)-quasi-
minimal points for \(g\) is an \emph{\(r\)-quasi-axis} for \(g\), which we denote \(A_{r}(g)\), i.e. \(A_{r}(g)=\{x \in X |d(x, g(x)) \leq \inf _{y \in X} d(y, g(y))+r\}\). If the
constant \(r\) is unimportant or clear from context, we may call \(A_{r}(g)\) simply a \emph{quasi-axis} for \(g\). If \(g\) is a loxodromic isometry on a $\delta$-hyperbolic space $X$, then
\(g\) fixes two points \(g^{+}\)and \(g^{-}\)on the boundary $\partial X$, where \(g^{+}\)is the attracting
fixed point of \(g\) and \(g^{-}\) is the repelling fixed point of \(g\).

Recall that an \emph{ending lamination} on a finite type surface is a geodesic lamination $\Lambda$ which is filling (every essential simple closed curve meets $\Lambda$, equivalently $S\setminus \Lambda$ is a union of disks and
once-punctured disks) and minimal, meaning that every leaf of $\Lambda$ is dense in $\Lambda$ (equivalently, $\Lambda$ has no proper nonempty closed
sublamination).
Observe that if $\Lambda$ is an ending lamination on $K_{1}$ and on $K_{2}$ (viewed as subsurfaces of the same ambient surface), then $K_{1}=K_{2}$ and both coincide with the support of $\Lambda$.
We will need the following lemma concerning elements that do not preserve $[K]$ and elements $K$-pseudo-Anosov mapping classes.

\begin{lemma}\label{1}
    Let $S$ be a connected orientable surface with positive complexity and $K \subset S$ be a nondisplaceable essential subsurface of finite type. Let \(\mathbb{X}\) be a hyperbolic space associated with $K$ as given by the Horbez-Qing-Rafi Theorem \cite{HQR20} on which \(\operatorname{Map}(S)\) admits a continuous nonelementary isometric action. Let $g \in \operatorname{Map}(S)$ be $K$-pseudo-Anosov. 
    
    Suppose $h \in \operatorname{Map}(S)$ does not preserve the isotopy class of $K$. Then $h$ does not preserve or flip the fixed points of $g$ on the boundary of $\mathcal{C}_{S}(K)$.
   
\end{lemma}

\begin{proof}
Suppose $h$ preserves the attracting fixed point of $g$ on the boundary of $\mathcal{C}_{S}(K)$. One has $h(g^{+})=g^{+}$. Masur and Minsky \cite{MM99} showed that the curve graph $\mathcal{C}_{S}(K)$ is $\delta$-hyperbolic and Klarreich \cite{K22} (See also \cite{HU06}) showed that the boundary of $\mathcal{C}_{S}(K)$ is homeomorphic to the space of ending laminations on $K$. It implies that $h$ preserves an ending lamination on $K$. It means that $h$ preserves its support $K$. This contradicts the fact that $h$ does not preserve the isotopy class of $K$. Similarly, it can be shown that $h$ does not preserve $g^{-}$.

If now $h(g^{+})=g^{-}$, then $h$ sends an ending lamination to another ending lamination with the same support. So $h$ preserves their support $K$, a contradiction.
\end{proof}

\begin{remark}
    The contrapositive is useful. If $h \in \operatorname{Map}(S)$ preserves a fixed point of a $K$-pseudo-Anosov element $g \in \operatorname{Map}(S)$ on the boundary of $\mathcal{C}_{S}(K)$, $h$ preserves the isotopy class of $K$.
\end{remark}

In a group $G$ acting by isometries on a geodesic $\delta$-hyperbolic space $X$, the following proposition extends the result of \cite[Proposition 2.1]{AD18} and provides a method to construct a free subgroup in $G$. The proof in
\cite[Proposition 2.1]{AD18} uses acylindricity to obtain a uniform bound on projections between quasi-axes of distinct conjugates. In our setting, we replace this input by the WWPD property of $g$, which provides the required bounded projection
constant. With this substitution, the ping-pong argument goes through unchanged.

\begin{proposition}\label{pp}
 Let $G$ be a group acting by isometries on a geodesic $\delta$-hyperbolic space $X$ and $h \in G$ be a non-trivial elliptic element for the action of $G$ on $X$. Assume that $g \in G$ is a WWPD loxodromic element on $X$ and $\operatorname{Fix}_{50\delta}(h) \cap A_{10\delta}(g)$ is bounded. Then there is a constant $C \geq 0$ such that for any integer $N>0$ for which the translation length of $g^{N}$ is at least $C$, the group generated by $g^{N}$ and $h$ is the free product of $\langle g^{N}\rangle$ and $\langle h\rangle$. Moreover, any elliptic subgroup of $\langle g^{N}, h\rangle$ is conjugate to a subgroup of $\langle h\rangle$.    
\end{proposition}

\begin{proof}
Recall that $\operatorname{Fix}_{50 \delta}(h)=\{x \in X| d(x, h(x)) \leq 50 \delta\}$. By assumption, there exist a constant $D \geq 0$ such that $\operatorname{Diam}(A_{10 \delta}(g) \cap \operatorname{Fix}_{50 \delta}(h) )\leq D.$ 

Given \(g, h \in G\), we use the conjugation notation \(g^{h}=h g h^{-1}\). Since $g \in G$ is a WWPD loxodromic element on $X$, there exists a constant $D^{\prime}$ such that for all \(h_{1}, h_{2} \in G\) satisfying
$\left(g^{h_{1}}\right)^{\pm} \neq \left(g^{h_{2}}\right)^{\pm},$
one has
$$
\operatorname{diam}\left(\pi_{A_{10 \delta}\left(g^{h_{2}}\right)}\left(A_{10 \delta}\left(g^{h_{1}}\right)\right)\right) \leq D^{\prime},
$$
where \(\pi_{A_{10 \delta}\left(g^{h_{2}}\right)}\left(A_{10 \delta}\left(g^{h_{1}}\right)\right)\) is the closest point
projection of \(A_{10 \delta}\left(g^{h_{1}}\right)\) onto \(A_{10 \delta}\left(g^{h_{2}}\right)\). This is implied by the discreteness discussed after Definition~\ref{DefWWPD}.

As $X$ is a geodesic $\delta$-hyperbolic space, there exist constants $K_{1}$, $K_{2}$ such that all $K_{1}$-local $(1, 10\delta)$-quasi-geodesics are $K_{2}$-global quasi-geodesics.

Set $\Delta= \operatorname{max}\{D, D^{\prime}, 1000 \delta, K_{1}, K_{2}^{2}\}$ and choose $C = 10\Delta+1000\delta$, then the natural map $\langle g^{N}\rangle * \langle h\rangle \to \langle g^{N}, h\rangle$ is an isomorphism, whose proof carries over word for word from \cite[Proposition 2.1]{AD18}.

If $\Gamma \leq \langle g^{N}, h\rangle$ contains an element $\gamma$ which is not conjugate into any subgroup of $\langle h\rangle$, then its cyclically reduced form contains at least one letter $g^{\pm N}$. Therefore, the stable translation length of $\gamma$ is strictly positive. It implies $\gamma$ is loxodromic. Consequently, if $\Gamma \leq \langle g^{N}, h\rangle$ is an elliptic subgroup, then $\Gamma$ contains no loxodromic elements, so $\Gamma$ must be contained in a conjugate of a subgroup of $\langle h\rangle$.

\end{proof}

\section{The proof of Theorem \ref{main}}

We now prove Theorem \ref{main}. Consider $h_1,\dots,h_n$ as in the statement, and let $K$ be a nondisplaceable essential finite type subsurface given by Lemma \ref{h}. In particular, no $h_i$ admits a representative restricting to the identity on $K$. Consider $\mathbb{X}$ a hyperbolic space associated with $K$ from the Horbez-Qing-Rafi Theorem.

We first prove the existence of infinitely many $K$-pseudo-Anosov elements that play ping pong with all the $h_i$ preserving the isotopy class of $K$. We then show that, by taking appropriate powers of these elements, we also obtain elements that play ping pong with every loxodromic or elliptic $h_i$ not preserving the isotopy class of $K$.

 \subsection{The case of elements preserving \texorpdfstring{$[K]$}{K}} 
 Assume that $h_1, \dots, h_n$ preserve the isotopy class of $K$. We always have $h(\mathcal{C}_{S}(K))=\mathcal{C}_{S}(K)$ and $h|_{\mathcal{C}_{S}(K)} \neq \mathrm{Id}_{\mathcal{C}_{S}(K)}$. Viewing $\operatorname{Map}(K, \partial K)$ as an acylindrical hyperbolic group for sufficiently complex $K$ (the genus $g(K) \geq 2$ or the boundary components $b(K) > 3$), $\operatorname{Map}(K, \partial K)$ has
property $P_{\text{naive}}$ \cite{AD18}: there exists a
$K$-pseudo-Anosov element $g$ that is a loxodromic element for the $\operatorname{Map}(S)$-action on $\mathbb{X}$ by Theorem \ref{WWPD} such that the pair $\{g, h_{i}\}$, $1\leq i\leq n$, generates a subgroup of $\mathrm{Map}(K, \partial K)$ which is a free product.
More precisely, there are infinitely many independent such $K$-pseudo-Anosov elements with disjoint pairs of fixed points \cite[Proposition 1.3]{AD18}.

\subsection{The case of loxodromic elements that do not preserve \texorpdfstring{$[K]$}{K}}

In this case, we will use the canonical ping-pong lemma to produce free subgroups in $\operatorname{Map}(S)$ (cf. \cite[Theorem 4.3.1]{Cla17}).

\begin{lemma}[Ping-Pong Lemma]\label{pingpong}
    Let $G$ be a group generated by two elements $a$ and $b$ of infinite order. Suppose there is a $G$-action on a set $X$ such that there are non-empty subsets $A, B \in X$ with $B$ not included in A  and such that for all $n \in \mathbb{Z} \backslash \{0\}$ we have $a^{n}(B) \subset A$ and $ b^{n}(A) \subset {B}.$ Then $G$ is freely generated by $\{a,b\}$. 
\end{lemma}
Recall that two loxodromic isometries $h$ and $g$ of a $\delta$-hyperbolic space $X$ are \emph{independent} if their fixed
point sets in $\partial X$ are disjoint, and the action of a group $G$ on $X$ is \emph{non-elementary} if $G$ contains two independent loxodromic isometries of $X$. If $h \in \operatorname{Map}(S)$ is loxodromic but does not preserve the isotopy class of $K$, then every $K$-pseudo-Anosov element $g\in\operatorname{Stab}_{\operatorname{Map}(S)}(K)$ is independent with $h$. By Lemma \ref{1}, $h$ does not preserve or flip the pair $(g^{+}, g^{-})$ on the boundary of $\mathcal{C}_{S}(K)$, so there exist two disjoint smaller neighborhoods \(U^{-}, U^{+}\subset \partial X\) of \(g^{-}, g^{+}\) respectively such that $h(U^{-}) \cap (U^{-}\cup U^{+})= \varnothing$ and $h(U^{+}) \cap (U^{-}\cup U^{+})= \varnothing$. Therefore, we have
$$ h\left(U^{+} \cup U^{-}\right)\subset \overline{\partial X \backslash  (U^{+} \cup U^{-})},$$ and the same inclusion for $h^{-1}$.
Finally, the dynamics of $g$ gives the following: for $M \geq 1$ and any integer $N \geq M$, we have
 $$g^{N}\left(\overline{\partial X \backslash U^{-}}\right) \subset U^{+} ~\text{and}~g^{-N}\left(\overline{\partial X \backslash U^{+}}\right) \subset U^{-}.$$ By Lemma \ref{pingpong}, take $A:= \overline{\partial X \backslash  (U^{+} \cup U^{-})}$ and $B:=U^{-}\cup U^{+}$, then $g^{N}$ and $h$ freely generate a subgroup of $\operatorname{Map}(S)$.

 \subsection{The case of elliptic elements that do not preserve \texorpdfstring{$[K]$}{K}}

Assume that $h \in \operatorname{Map}(S)$ is elliptic but does not preserve the isotopy class of $K$. Every $K$-pseudo-Anosov elements $g$ is a WWPD loxodromic element for the $\operatorname{Map}(S)$-action on $\mathbb{X}$ (Theorem \ref{WWPD}). By Lemma \ref{1}, there exists a constant $D \geq 0$ satisfying $\operatorname{diam}(\operatorname{Fix}_{50\delta}(h) \cap A_{10\delta}(g)) \leq D$. By Proposition \ref{pp}, there is a constant $C \geq 0$ such that for any integer $N$, the translation length of $g^{N}$ is at least $C$ and the group generated by $g^{N}$ and $h$ is the free product of $\langle g^{N} \rangle$ and $\langle h \rangle$.

\subsection{The case of parabolic elements that do not preserve \texorpdfstring{$[K]$}{K}}
Recall $\mathbb{X}$ is a quasi-tree of curve graphs $\{\mathcal{C}_{S}(K_{1})| K_{1} \in (\operatorname{Map}(S) \cdot K)\}$. Let $\dot{\mathbb{X}}$ be the projection quasi-tree by collapsing each curve graph $\mathcal{C}_{S}(K_{1})$ to a vertex $V_{K_{1}}$. One can describe the boundary $\partial{\mathbb{X}}$ as the union of $\partial{\dot{\mathbb{X}}}$ and all boundaries $\partial{\mathcal{C}_{S}(K_{1})} $ for $K_{1} \in (\operatorname{Map}(S) \cdot K)$.

Assume that $h$ is parabolic. Then there exist $\xi \in \partial \mathbb{X}$ and $v \in \mathbb{X}$ such that
$h^{n}(v) \to \xi$ as $n \to \infty$, and $h(\xi)=\xi$. Observe that $h$ induces an isometry
$\dot{h} \in \operatorname{Isom}(\dot{\mathbb{X}})$. Since $h$ is parabolic, we have
\(
\frac{1}{n} d_{\mathbb{X}}(h^{n}(v),v)\rightarrow 0~ \text{as} ~ n\to\infty.
\)
Hence, in $\dot{\mathbb{X}}$ we also have
\(
\frac{1}{n} d_{\dot{\mathbb{X}}}(\dot{h}^{n}(\dot{v}),\dot{v})\rightarrow 0,
\)
so $\dot{h}$ is either elliptic or parabolic. By \cite{Man06}, every element of a group acting on a
quasi-tree acts either loxodromically or elliptically. Therefore, $\dot{h}$ is elliptic and it has a bounded orbit in $\dot{\mathbb{X}}$. It follows that $\xi \notin \partial \dot{\mathbb{X}}$, and thus there exists
$K_{1} \in (\operatorname{Map}(S)\cdot K)$ such that $\xi \in \partial \mathcal{C}_{S}(K_{1})$.
By the ending lamination argument (see Lemma \ref{1}) and the fact that $h(\xi)=\xi$, the map $h$ preserves the isotopy
class of $K_{1}$ and induces an isometry of $\mathcal{C}_{S}(K_{1})$. However, no element of
$\operatorname{Map}(K_{1})$ acts parabolically (\cite[Corollary 1.3]{BG18}), while its orbits
accumulate on $\partial \mathcal{C}_{S}(K)$, contradicting our assumption. Therefore, there is no
element of $\operatorname{Map}(S)$ that acts parabolically on $\mathbb{X}$.

Now we give the proof of Theorem \ref{main}. 
 
\begin{proof}[Proof of Theorem \ref{main}.]
Let \( F = \{h_1, h_2, \dots, h_n\} \) be a finite collection of non-trivial elements of \( \operatorname{Map}(S) \). For each \( h_i \in F \), where \( 1 \leq i \leq n \), we analyze the element \( h_i \) via the isometric action of \( \operatorname{Map}(S) \) on \( \mathbb{X} \). Each \( h_i \) is either elliptic or loxodromic. In any case, there exists a \( K \)-pseudo-Anosov element in \( \operatorname{Map}(S) \) for the action on \( \mathbb{X} \), which only depends on $K$ and works for any $h \in \operatorname{Map}(S)$ such that the restriction $h|_{K} \neq \mathrm{Id}$.

Thus, for all \( i \), we have \( \langle g, h_i \rangle \cong \langle g \rangle * \langle h_i \rangle \), meaning the group generated by \( g \) and \( h_i \) is isomorphic to the free product of \( \langle g \rangle \) and \( \langle h_i \rangle \).    
\end{proof}

\bibliographystyle{alpha}
\bibliography{sample}

\end{document}